\documentclass[11pt]{article}

\usepackage[utf8]{inputenc}
\usepackage[T1]{fontenc}
\usepackage{mathpazo}

\RequirePackage[english]{babel}

\RequirePackage[a4paper,margin=2.5cm]{geometry}
\RequirePackage{parskip}

\newcommand{\affiliation}{\footnote}
\makeatletter
\def\@fnsymbol#1{\ensuremath{\ifcase#1\or *\or \dagger\or \ddagger\or \mathsection\or \|\or **\or \dagger\dagger \or \ddagger\ddagger \else\@ctrerr\fi}}
\makeatother

\usepackage{xcolor}
\definecolor{cblue}{RGB}{0,70,140}
\definecolor{cgreen}{RGB}{100,140,0}
\definecolor{cred}{RGB}{190,10,50}

\usepackage[shortlabels]{enumitem}
\setlist{itemsep=0ex,topsep=0ex,parsep=0.4ex}

\usepackage{amsmath,amsfonts,amssymb,amsthm,mathtools,thmtools,thm-restate,bbm,centernot}

\usepackage[hyphens]{url} 
\usepackage[hidelinks,colorlinks,pagebackref]{hyperref} 
\hypersetup{citecolor=cgreen,linkcolor=cblue,urlcolor=cblue}
\usepackage[capitalise,nameinlink,noabbrev,compress]{cleveref} 

\renewcommand*{\backref}[1]{}
\renewcommand*{\backrefalt}[4]{
	\ifcase #1 Not cited.%
	\or $\uparrow$#2%
	\else $\uparrow$#2%
	\fi%
}

\theoremstyle{plain}
\newtheorem{theorem}{Theorem}[section]
\newtheorem{lemma}[theorem]{Lemma}

\theoremstyle{definition}

\makeatletter
\renewenvironment{proof}[1][\proofname]
{\par\pushQED{\qed}
	\normalfont\topsep6\p@\@plus6\p@\relax\trivlist
	\item[\hskip\labelsep\bfseries#1\@addpunct{.}]
	\ignorespaces}
{\popQED\endtrivlist\@endpefalse}
\makeatother

\newcommand{\eps}{\varepsilon}

\newcommand{\bN}{\mathbb{N}}
\newcommand{\cO}{\mathcal{O}}
\newcommand{\dcol}{d_{\textrm{col}}}
\newcommand{\val}{w}
\DeclarePairedDelimiter{\abs}{\lvert}{\rvert}
\DeclarePairedDelimiter{\ceil}{\lceil}{\rceil}

\title{An improved upper bound for the multicolour Ramsey number of odd cycles}
\author{Maria Axenovich\affiliation{Institute of Algebra and Geometry, Karlsruhe Institute of Technology, Germany (\textsf{\href{mailto:maria.aksenovich@kit.edu}{maria.aksenovich@kit.edu}}).} \and Wouter Cames van Batenburg\affiliation{D\'epartement d'Informatique, Universit\'e libre de Bruxelles, Belgium (\textsf{\href{mailto:w.p.s.camesvanbatenburg@gmail.com}{w.p.s.camesvanbatenburg@gmail.com}}). Supported by the Belgian National Fund for Scientific Research (FNRS).} \and Oliver Janzer\affiliation{Institute of Mathematics, EPFL, Lausanne, Switzerland (\textsf{\href{mailto:oliver.janzer@epfl.ch}{oliver.janzer@epfl.ch}}).} \and Lukas Michel\affiliation{Mathematical Institute, University of Oxford, United Kingdom (\textsf{\href{mailto:lukas.michel@maths.ox.ac.uk}{lukas.michel@maths.ox.ac.uk}}).} \and Mathieu Rundstr\"{o}m\affiliation{Department of Combinatorics and Optimization, University of Waterloo, Waterloo, Canada (\textsf{\href{mailto:mrundstrom@uwaterloo.ca}{mrundstrom@uwaterloo.ca}}).}}
\date{20 October 2025}

\begin{document}
    \maketitle

    \begin{abstract}
        We show that the $k$-colour Ramsey number of an odd cycle of length $2 \ell + 1$ is at most $(4 \ell)^k \cdot k^{k/\ell}$. This proves a conjecture of Fox and is the first improvement in the exponent that goes beyond an absolute constant factor since the work of Bondy and Erd\H{o}s from 1973.
    \end{abstract}
    
    \section{Introduction}
    
    The \emph{$k$-colour Ramsey number} $R_k(H)$ of a graph $H$ is the smallest integer $n$ such that every $k$-edge-colouring of the complete graph $K_n$ contains a monochromatic copy of $H$. For the triangle, the notoriously difficult Schur-Erd\H{o}s problem asks to determine the growth rate of $R_k(C_3)$. In 1916, Schur \cite{S16} showed that
    \[
        \Omega(2^k) \le R_k(C_3) \le \cO(k!).
    \]
    Since then, the upper bound has remained unchanged up to small improvements to the constant factor, while the lower bound has been improved to $\Omega(3.28^k)$ by Ageron, Casteras, Pellerin, Portella, Rimmel, and Tomasik \cite{ACPPRT21}. Erd\H{o}s conjectured that $R_k(C_3) = 2^{\Theta(k)}$ \cite{CG98} and offered monetary awards for proving this conjecture and solving some related problems.
    
    For longer odd cycles, Bondy and Erd\H{o}s~\cite{BE73} and Erd\H{o}s and Graham \cite{EG73} obtained the bounds
    \[
        \ell \cdot 2^k + 1 \le R_k(C_{2\ell+1}) \le 2 \ell \cdot (k+2)!.
    \]
    If $k$ is fixed, it turns out that the lower bound is sharp as Jenssen and Skokan~\cite{JS21} proved that $R_k(C_{2\ell+1}) = \ell \cdot 2^k + 1$ for all sufficiently large $\ell$. However, this is not true if $\ell$ is fixed. Indeed, Day and Johnson \cite{DJ17} showed that for all $\ell$ there exists some $\delta \coloneqq \delta(\ell) > 0$ such that $R_k(C_{2\ell+1}) \ge 2 \ell \cdot (2+\delta)^{k-1}$ for all sufficiently large $k$.

    Usually, if the number of colours is large, longer odd cycles should be easier to find than shorter odd cycles. For instance, Fox \cite{F} conjectured that for every $\eps > 0$ there exists some $\ell$ such that $R_k(C_{2\ell+1}) \le k^{\eps k}$ for all sufficiently large $k$. Li \cite{Li09} even made the stronger conjecture that $R_k(C_{2\ell+1}) \le o(k!^{1/\ell})$ as $k \to \infty$. Nevertheless, the gap between these conjectures and the best upper bounds known remained large. Li \cite{Li09} proved that $R_k(C_5) \le c^k \cdot \sqrt{k!}$ for some constant $c$, which was later extended by Lin and Chen \cite{LC19} to longer odd cycles by showing that for $\ell \ge 2$ we have $R_k(C_{2\ell+1}) \le c^k \cdot \sqrt{k!}$ for some constant $c$ that only depends on $\ell$.\footnote{These results were proved independently by Fox \cite{F}, but were never published.} Only under the wide open additional assumption that each Ramsey graph for $R_k(C_{2\ell+1})$ is nearly regular\footnote{That is, there is an absolute constant $\eps$ such that for all sufficiently large $k$, the minimum degree of every Ramsey graph for $R_k(C_{2\ell+1})$ is at least an $\eps$-fraction of its average degree.}, Li \cite{Li09} showed that this bound could be improved to $R_k(C_{2\ell+1}) \le c^k \cdot k!^{1/\ell}$.
    
    In this paper, we prove the conjecture of Fox.

    \begin{theorem}\label{thm:multicolourramseyoddcycle}
        For $k,\ell \in \bN$,
        \[
            R_k(C_{2\ell+1}) \le (4 \ell - 2)^k \cdot k^{k/\ell} + 1.
        \]
    \end{theorem}

    Using the well-known inequality $(k/e)^k < k!$, this bound implies that $R_k(C_{2\ell+1}) \le c^k \cdot k!^{1/\ell} + 1$ for $c \coloneqq (4 \ell - 2) e^{1/\ell}$, and so this establishes the conditional upper bound of Li unconditionally.
    
    In addition to finding a monochromatic odd cycle of a specific length, there has also been considerable interest in finding any short monochromatic odd cycle. It is easy to see that there is a $k$-edge-colouring of $K_{2^k}$ without any monochromatic odd cycles, but that every $k$-edge-colouring of $K_{2^k+1}$ contains such a cycle. Motivated by this, in 1973, Erd\H{o}s and Graham~\cite[Question (iii) in Section 6]{EG73} asked for the smallest integer $L(k)$ such that every $k$-edge-colouring of $K_{2^k+1}$ contains a monochromatic odd cycle of length at most $L(k)$.
    
    Day and Johnson~\cite{DJ17} proved that $L(k) \ge 2^{\Omega(\sqrt{\log k})}$. Recently, Gir\~ao and Hunter~\cite{GH24} obtained the first non-trivial upper bound, showing that $L(k) \le (2^k+1) / k^{1-o(1)}$. Using an algebraic approach, Janzer and Yip~\cite{JY25} improved this to $L(k) \le \cO(k^{3/2} \cdot 2^{k/2})$.
    
    Both Gir\~ao and Hunter \cite{GH24} and Janzer and Yip \cite{JY25} also discussed the more general problem of finding short monochromatic odd cycles in colourings of complete graphs with more vertices. In the regime where the number of vertices is $(2 + \delta)^k$ for some small $\delta > 0$, the methods of both papers can guarantee a monochromatic odd cycle of length at most $\cO_\delta(k)$. The dependence on $\delta$ is better in \cite{JY25}, where the length of the cycle is $\cO(\delta^{-1/2} \cdot k)$ if $\delta < 1$. With our method, we can find significantly shorter monochromatic odd cycles, unless $\delta$ tends to $0$ very quickly.

    \begin{theorem}\label{thm:shortmonochromaticoddcycle}
        For $k \in \bN$ and $b > 2$, every $k$-edge-colouring of $K_n$ with $n > b^k$ contains a monochromatic odd cycle of length at most $2 \ceil{\log_{b/2} k} + 1$.
    \end{theorem}
    
    In particular, if $b = 2 + \delta$ for $0 < \delta < 1$, this result guarantees a cycle of length $\cO(\delta^{-1} \cdot \log k)$. In addition, the result also applies to larger values of $b$. If $b = k^\eps$, this yields a result similar to \cref{thm:multicolourramseyoddcycle}, but without controlling the exact cycle length. We remark that \cref{thm:shortmonochromaticoddcycle} does not give any non-trivial result for the Erd\H os--Graham problem, where $\delta \approx 1 / (k \cdot 2^{k-1})$.
    
    \textbf{Notation.} For an edge-coloured graph, let $N_c^i(v)$ be the set of all vertices $u$ such that the shortest path of colour $c$ from $v$ to $u$ has length $i$, and write $N_c(v) \coloneqq N_c^1(v)$ and $N_c^{\le \ell}(v) \coloneqq \bigcup_{i=0}^\ell N_c^i(v)$. For an uncoloured graph, $N^i(v)$ denotes the set of vertices at distance exactly $i$ from $v$.

    \section{Neighbourhoods with small chromatic number}

     In~\cref{sec:corollaries}, we prove \cref{thm:multicolourramseyoddcycle,thm:shortmonochromaticoddcycle} using the following key lemma. This result bounds the number of vertices of a $k$-edge-coloured complete graph as long as for every vertex $v$ and every colour $c$, the subgraph of colour $c$ induced by the neighbourhood $N_c^{\le \ell}(v)$ has small chromatic number (which is the case if the colouring does not contain a monochromatic odd cycle of length $2\ell+1$). In fact, we prove this lemma for \emph{$k$-local-edge-colourings}, which are edge-colourings where at most $k$ colours are incident to each vertex.
    
    \begin{lemma}\label{thm:neighbourhoodchromaticsizebound}
        Let $k, \ell, \chi \in \bN$. Consider a $k$-local-edge-colouring of a complete graph $K_n$ such that for every vertex $v \in V(K_n)$ and every colour $c$, the subgraph of colour $c$ induced by $N_c^{\le \ell}(v)$ in $K_n$ has chromatic number at most $\chi$. Then, $n \le \chi^k \cdot k^{k/\ell}$.
    \end{lemma}

    \begin{proof}
        Let $G \coloneqq K_n$. Define the weight of a vertex $v \in V(G)$ as $\val(v) \coloneqq (\chi \cdot k^{1/\ell})^{-\dcol(v)}$ where $\dcol(v)$ denotes the number of colours incident to $v$, and define the weight of a subset of vertices $U \subseteq V(G)$ as $\val(U) \coloneqq \sum_{v \in U} \val(v)$. Note that $\val(v) \ge \chi^{-k} \cdot k^{-k/\ell}$ for every vertex $v \in V(G)$, and so $\val(V(G)) \ge n \cdot \chi^{-k} \cdot k^{-k/\ell}$. To show that $n \le \chi^k \cdot k^{k/\ell}$, it therefore suffices to prove that $\val(V(G)) \le 1$. We will prove this by induction on the number of vertices of $G$. The case $\abs{V(G)} = 1$ is trivial, so suppose that $\abs{V(G)} \ge 2$.
    
        Let $v \in V(G)$ be arbitrary. Since at most $k$ colours are incident to $v$, there exists a colour $c$ such that $\val(N_c(v)) \ge \val(V(G) \setminus \{v\})/k$, and so $\val(N_c^{\le 1}(v)) = \val(\{v\} \cup N_c(v)) > \val(V(G))/k$. We claim that there exists some $i \in [\ell]$ such that $\val(N_c^{i+1}(v)) \le (k^{1/\ell} - 1) \cdot \val(N_c^{\le i}(v))$. Indeed, otherwise we have $\val(N_c^{\le i+1}(v)) = \val(N_c^{\le i}(v)) + \val(N_c^{i+1}(v)) \ge k^{1/\ell} \cdot \val(N_c^{\le i}(v))$ for all $i \in [\ell]$ and so $\val(N_c^{\le \ell+1}(v)) \ge (k^{1/\ell})^\ell \cdot \val(N_c^{\le 1}(v)) = k \cdot \val(N_c^{\le 1}(v)) > \val(V(G))$, a contradiction.

        Let $S \coloneqq N_c^{\le i}(v)$ and $T \coloneqq N_c^{i+1}(v)$. In particular, $\val(T) \le (k^{1/\ell} - 1) \cdot \val(S)$. By assumption, the subgraph of colour $c$ induced by $S$ in $G$ has chromatic number at most $\chi$. Fix such a $\chi$-vertex-colouring of $S$ and let $S' \subseteq S$ be one of its colour classes with maximum weight. Then, $S'$ spans no edge of colour $c$ and satisfies $\val(S') \ge \val(S)/\chi$.

        Delete all vertices of $T \cup (S \setminus S')$ from $G$ and update the weights of the remaining vertices. Note that the weight of every vertex in $S'$ increases by a multiplicative factor of at least $\chi \cdot k^{1/\ell}$ since colour $c$ is no longer incident to any of these vertices. On the other hand, the weight of every vertex in $V(G) \setminus (T \cup S)$ is either unchanged or increases. Therefore, the weight of the entire graph increases by at least
        \[
            (\chi \cdot k^{1/\ell}) \cdot \val(S') - \val(S) - \val(T) \ge k^{1/\ell} \cdot \val(S) - \val(S) - (k^{1/\ell} - 1) \cdot \val(S) = 0.
        \]
        Since we know by the induction hypothesis that the weight of the new graph is at most $1$, it follows that the weight of the original graph was at most $1$.
    \end{proof}

    \section{Consequences for short monochromatic odd cycles}\label{sec:corollaries}

    To deduce \cref{thm:multicolourramseyoddcycle,thm:shortmonochromaticoddcycle} for cycles of length at most $2\ell +1$, it remains to bound the chromatic number of the neighbourhoods $N_c^{\le \ell}(v)$. This is very easy for \cref{thm:shortmonochromaticoddcycle} since none of the neighbourhoods $N_c^i(v)$ for $i \le \ell$ can span an edge of colour $c$.

    \begin{proof}[Proof of \cref{thm:shortmonochromaticoddcycle}]
        Let $\ell \coloneqq \ceil{\log_{b/2} k}$. Note that if a $k$-edge-colouring of $K_n$ contains no monochromatic odd cycle of length at most $2 \ell + 1$, then for every vertex $v \in V(K_n)$, every colour $c$, and every $i \in [\ell]$ it holds that $N_c^i(v)$ spans no edge of colour $c$. This implies that the subgraph of colour $c$ induced by $N_c^{\le \ell}(v)$ in $K_n$ is bipartite. By \cref{thm:neighbourhoodchromaticsizebound}, it follows that $n \le 2^k \cdot k^{k/\ell} \le 2^k \cdot (b/2)^k = b^k$.
    \end{proof}

    To prove \cref{thm:multicolourramseyoddcycle}, we use a known bound on the chromatic number of the subgraph induced by $N^i(v)$ for $i \le \ell$ in graphs that contain no cycle of length $2 \ell + 1$.

    \begin{proof}[Proof of \cref{thm:multicolourramseyoddcycle}]
        An argument of Erd\H{o}s, Faudree, Rousseau, and Schelp~\cite{EFRS78} shows that if a graph $G$ contains no cycle of length $2 \ell + 1$, then for every vertex $v \in V(G)$ and every $i \in [\ell]$ it holds that $G[N^i(v)]$ has chromatic number at most $2 \ell - 1$. So, if a $k$-edge-colouring of $K_n$ contains no monochromatic cycle of length $2 \ell + 1$, then for every vertex $v \in V(K_n)$ and every colour $c$, the subgraph of colour $c$ induced by $N_c^{\le \ell}(v)$ in $K_n$ has chromatic number at most $4 \ell - 2$. By \cref{thm:neighbourhoodchromaticsizebound}, it follows that $n \le (4 \ell - 2)^k \cdot k^{k/\ell}$.
    \end{proof}

    \textbf{Acknowledgements.} We thank David Conlon and Jacob Fox for their helpful comments and for pointing out some relevant references. This work was initiated at the ``Topics in Ramsey theory'' online workshop of the Sparse Graphs Coalition. We thank Stijn Cambie, Nemanja Dragani\'c, Ant\'onio Gir\~ao, Eoin Hurley, and Ross Kang for organising this event.
    
    {
    \hypersetup{linkcolor=cred}
    \newcommand{\etalchar}[1]{$^{#1}$}
    
    }
\end{document}